\documentclass[11pt, reqno]{amsart}
\usepackage{geometry}                
\usepackage{subcaption}
\usepackage{cleveref}                
\usepackage{graphicx}
\usepackage{amssymb}
\usepackage{epstopdf}
\newcommand{\bd}[1]{\mathbf{#1}}  
\newcommand{\RR}{\mathbb{R}}      
\newcommand{\ZZ}{\mathbb{Z}}      

\newcommand{\mat}[1]{\left(\begin{matrix} #1 \end{matrix} \right)}  
\newcommand{\al}[1]{\begin{align}#1\end{align}}
\newcommand{\aln}[1]{\begin{align*}#1\end{align*}}

\newcommand{\DD}{\mathbb{D}}

\newcommand{\Aa}{\mathcal{A}}
\newcommand{\ee}{\mathbb{E}}

\newtheorem{thm}{Theorem}[section]
\newtheorem{lem}[thm]{Lemma}

\newtheorem{cor}[thm]{Corollary}

\title{The Gap Distribution of Directions in Some Schottky Groups}
\author{Xin Zhang}
\address{1409 W Green Street, Urbana IL 61801, United States}
\email{xz87@illinois.edu}

\begin{document}
\maketitle
\begin{abstract}
We prove the existence and some properties of the limiting gap distribution functions for the directions of some thin group orbits in the Poincar\'e disk.   
\end{abstract}

\section{Introduction}

In this paper,  we investigate the limiting gap distribution of directions for some thin groups.  Specifically, we let $\Gamma_0$ be the Schottky group generated by three reflections $\rho_1,\rho_2,\rho_3$, with non-intersecting isometry half-circles $C_1,C_2,C_3$ in the Poincar\'e disk $\DD$. The region $\mathcal{F}$ bounded by $C_1,C_2,C_3$ is a fundamental domain for $\Gamma_0$.  $\Gamma_0$ is called a \emph{thin group} because $\mathcal{F}$ has infinite volume.  We assume $\mathcal{F}$ is \emph{principal}, in the sense that the Jacobians $|\rho_1^{'}|,|\rho_2^{'}|,|\rho_3^{'}| \leq 1$ on $\mathcal{F}$.  Every such group has a unique principal fundamental domain $\mathcal{F}$, except there are two when one of the isometry circles is a line passing 0. We do not deal with the latter case, although our method can be easily modified to cover this case. Let $z$ be a point on $\partial \DD$, $\mathcal{B}_T$ be the set of elements $\gamma$ in $\Gamma_0$ with $||\gamma||<T$ for some norm $||\cdot||$.  We put an observer at the origin 0, and we want to study the limiting behavior of the gap distribution of directions, or angles of $\{\mathcal{B}_T (z)\}$ when observing at the origin, as $T$ goes to infinity. \\

  \begin{figure}[htbp] 
    \centering
    \includegraphics[width=2.5in]{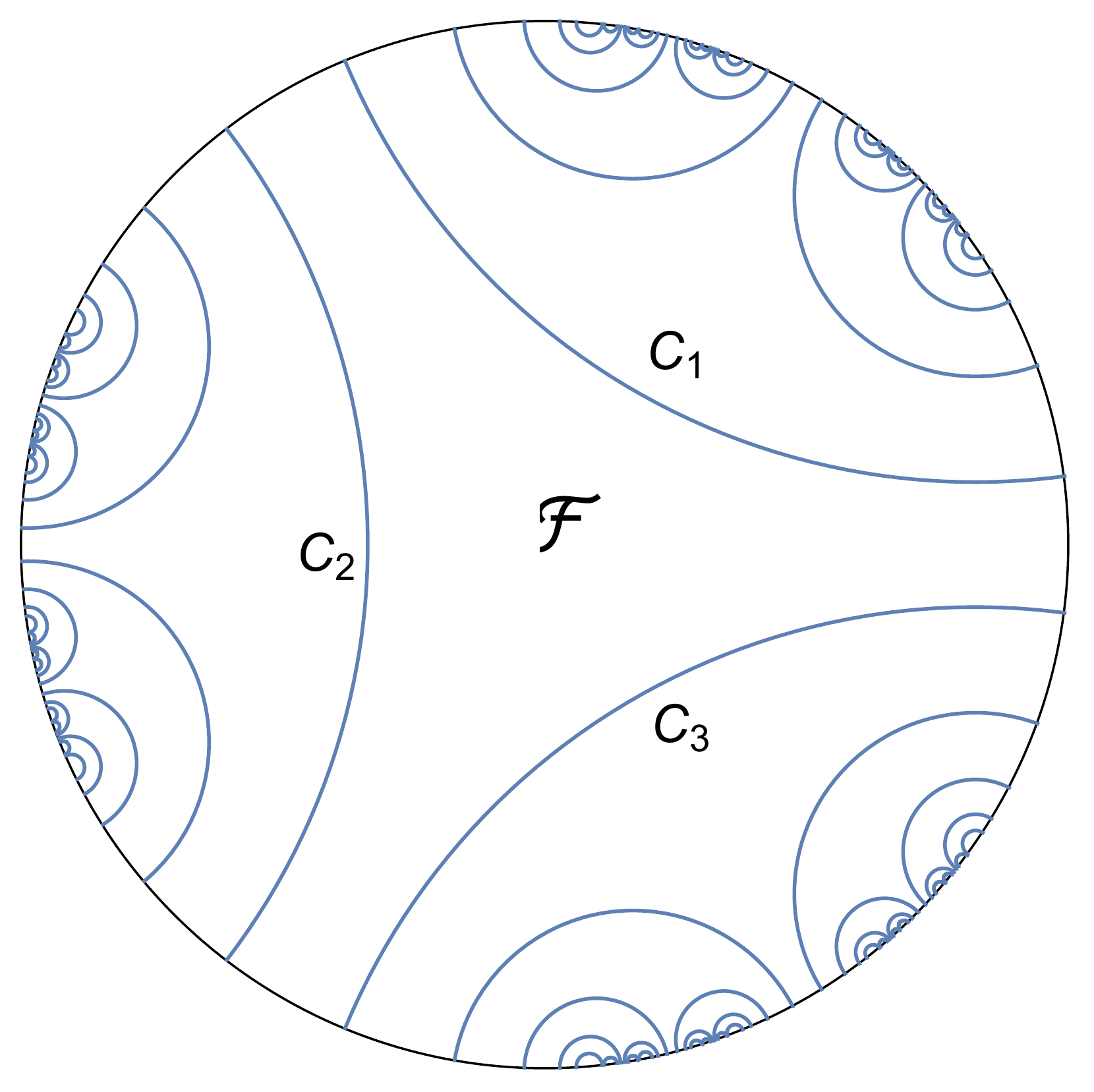} 
   \caption{A hyperbolic reflection group}
    \label{fig1}
 \end{figure}

The group $\Gamma_0$ has an index-2, orientation-preserving and free subgroup $\Gamma$ generated by $\rho_1\rho_2,\rho_1\rho_3$.  As a subgroup of a thin group $\Gamma_0$, $\Gamma$ is also thin.  If $\Gamma^{'}$ is of finite covolume, or a lattice in $SL(2,\RR)$, Kelmer and Kontorovich \cite{KK15} proved the limiting pair correlation of the directions of a single $\Gamma^{'}$-orbit in $\DD$.  This generalizes the work of Boca, Popa, and Zaharescu \cite{BPZ14} which deals with the case when $\Gamma^{'}$ is $SL(2,\ZZ)$ and the observer is placed at an elliptic point.  Later, the work \cite{KK15} was further generalized by Risager and 
S\"odergren  \cite{RS14} to cover the cases $SO(n,1)$ with explicit convergence rate, and by Marklof and Vinogradov \cite{MV14} which determines a large class of local statistics, including gap distribution.  \\

Each of \cite{KK15}, \cite{RS14} and \cite{MV14} used some automorphic tool (either homogeneous dynamics or spectral theory).  Applying automorphic tool to local statistics studies dates back at least to \cite{Ma00}.  See also \cite{EM04}, \cite{MS10}, \cite{RZ15}, \cite{AC10}, \cite{AC14} for some interesting works of this flavor.\\

In fact, all these works and other known works on local statistics related to orbits of discrete hyperbolic isometry groups, have the underlying groups finite co-volume.  Therefore, the novelty of this paper is that it deals with some thin group orbits.  In doing so, we exhibit both similarities and dissimilarities with lattice cases, furthermore, our work serves to give a general flavor for this type of problems in the thin-group case.\\

 \begin{figure}[htbp] 
    \centering
    \includegraphics[width=2.5in]{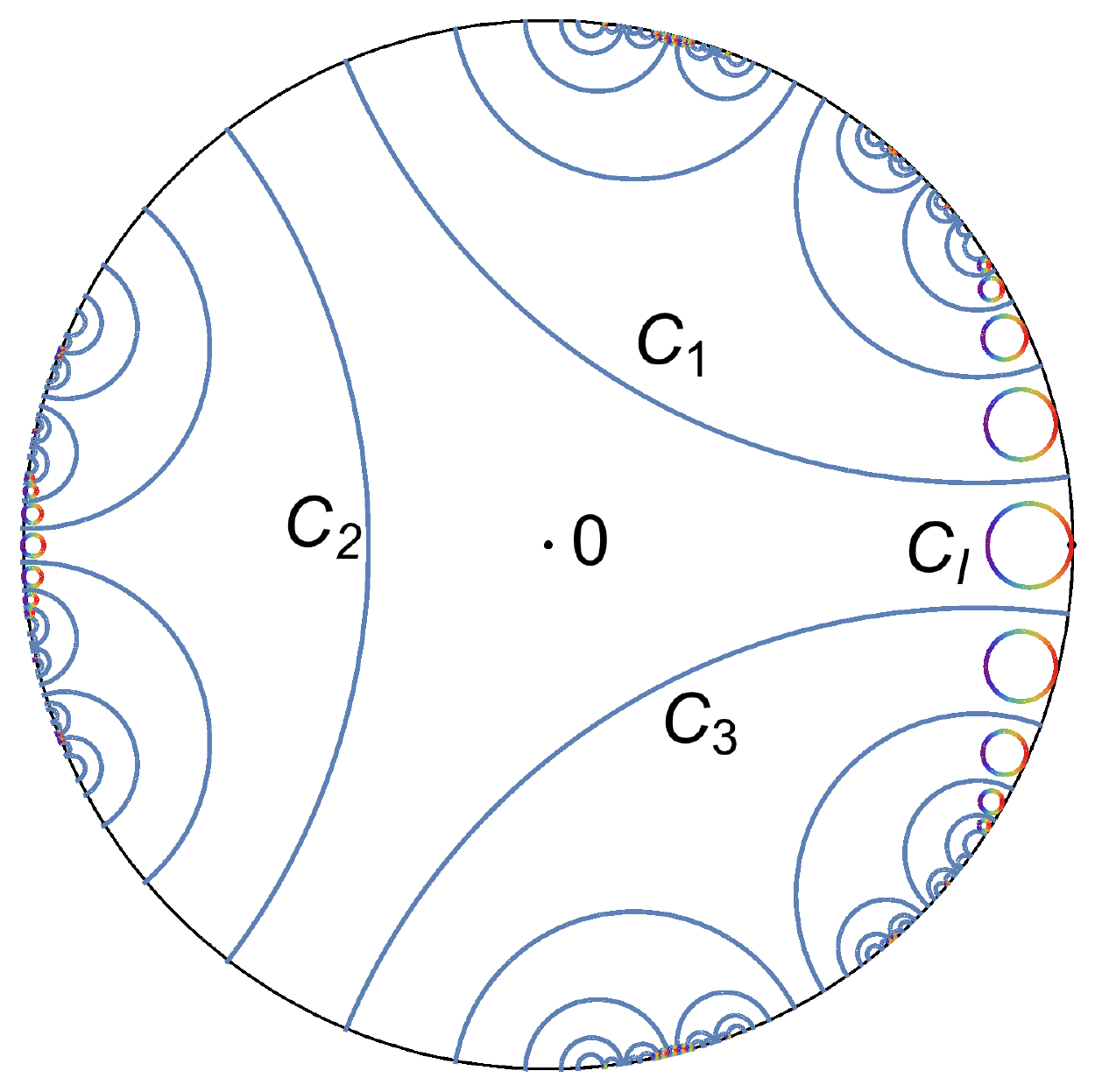} 
   \caption{The family of $C_{\gamma}$'s (hued circles)}
    \label{fig2}
 \end{figure}

Compared to \cite{KK15},\cite{MV14}, we modify our problem in the following way:\\\\
(1) As mentioned above, instead of choosing $z$ inside $\DD$, we choose $z$ to be a point on the boundary, say 1, and 1 lies in the flare region between $C_1$ and $C_3$ (see Figure \ref{fig2}). \\
(2) In \cite{KK15} and \cite{MV14}, they use the Frobenious norm, given by 
$$\Big\vert\Big\vert\mat{a&b\\c&d}\Big\vert\Big\vert_{\rm{Frob}}^2=\frac{|a|^2+|b|^2+|c|^2+|d|^2}{2}.$$
We use a different norm $||\cdot||$. We put a tiny circle $C_{I}=C(1-r_0,r_0)$ tangent at 1 (See Figure \ref{fig2}).  Let $C_{\gamma}=\gamma(C_I)=C((1-r_{\gamma})e^{\theta_{\gamma}\bd{i}}, r_{\gamma})$. We define 
$$||\gamma||^2:=\kappa(C_\gamma)=\frac{1}{r_{\gamma}},$$
where $\kappa(C_\gamma)$ is the curvature, or the reciprocal of the radius of $C_\gamma$.  It turns out that $||\cdot||$ and $||\cdot||_{\rm{Frob}}$ are equivalent, when restricted to $\Gamma$ (see Lemma \ref{equivalentnorms}).  We choose $r_0$ in the following way.  Let $d_E(x,y)$ be the Euclidean distance between $x$ and $y$.  Each $C_i(i=1,2,3)$ divides the circle $\partial\DD$ into a major arc $\mathfrak{M}_i$ and a minor arc $\mathfrak{m}_i$.  let $D_i$ be the region in $\DD$ bounded by $C_i$ and $\mathfrak{m}_i$, we then define the Euclidean distance between the two sets $D_i$ and $D_j$ (we still use the notation $d_E$) by
$$d_E(D_i,D_j)=\inf_{x,y}\{d_E(x,y)\vert x\in D_i,y\in D_j\},$$
We choose $r_0$ so that $C_I\subset {\mathcal{F}}\cup\{1\}$ and $r_0<\frac{1}{3}\min_{i\neq j}\{d_E({D_i,D_j})\}$.\\

The advantage of the above modification is twofold: first the directions of orbit points are just tangencies $\{\gamma (1)\}$, which simplifies our calculation; second the gaps between these tangencies can be explicitly parametrized by $\Gamma$.\\

Unlike a lattice point orbit which converges to everywhere on the boundary,  the orbit $\Gamma z$ converges to a limit set $\Lambda_\Gamma$ in $\partial\DD$, which is a Cantor-like, totally disconnected, perfect set with Hausdorff dimension $<1$.  Let $\delta$ be the critical exponent for $\Gamma$ (roughly speaking, $\delta$ measures the asymptotic growth rate of $\Gamma$).  When $\delta>1/2$, $\delta$ also agrees with the Hausdorff dimension of the limit set $\Lambda_\Gamma$ (See Page 12 and Page 32 of \cite{Mc98} for the computation of $\delta$ for some of such groups).
 Supported on $\Lambda_\Gamma$, there is a canonical $\Gamma$-invariant geometric measure of dimension $\delta$, called the Patterson-Sullivan measure, denoted by $\mu$. $\mu$ is a probability measure and has no atom \cite{Pa76}\cite{Su79}\cite{Su84}. \\

Now we introduce our problem.  Let $\mathcal{I}$ be an interval of $\partial \DD$ with $\mu(\mathcal{I})>0$.  Let $\mathcal{A}_T=\mathcal{A}_{T,\mathcal{I}}=\{x_i\}$ be the counter-clockwise oriented sequence of the points $$\left\{\gamma(1)\in\mathcal{\mathcal{I}} \big\vert\gamma\in\Gamma_0, ||\gamma||<T\right\}.$$  The gap between $x_i$ and $x_{i+1}$, denoted by $d(x_i,x_{i+1})$, is just the arclength distance between $x_i$ and $x_{i+1}$.  The size of $\Aa_T$ is asymptotically $c_0\mu(\mathcal{I}) T^{2\delta}$ for some positive constant $c_0$ (Theorem \ref{countpoints}).  Therefore,  when defining the gap distribution function, a first thought is to normalize the gaps in $\Aa_T$ by dividing them by $T^{2\delta}$, as in the previous works.   However, as the following Theorem \ref{mainthm} indicates, it turns out most gaps are not of the order of the average gap ${1}/{T^{2\delta}}$, but of the order ${1}/{T^2}$.  And unlike a typical lattice orbits (for example, Farey sequences), any gap in $\mathcal{A}_T$ has two end points $\gamma_1(1),\gamma_2(1)$ with $||\gamma_1||,||\gamma_2||$ comparable (see Theorem \ref{comparable}). \\

We define our gap distribution function to be 
\al{\label{11082323} F_{T,\mathcal{I}}(s)=\frac{1}{c_0\mu(\mathcal{I})T^{2\delta}}\sum_{x_i\in\Aa_T}\bd{1}\{\frac{d(x_i,x_{i+1})}{T^2}\leq s\}.}

Our main theorem is the following:
\begin{thm} \label{mainthm}There exists a monotone, continuous function $F$ on $(0,\infty)$, independent of $\mathcal{I}$, such that 
$$\lim _{T\rightarrow \infty} F_{T,\mathcal{I}}(s)=F(s).$$
Moreover, $F$ is supported away from 0,
and $$\lim_{s\rightarrow \infty}F(s)=1.$$
\end{thm}

Figure \ref{dist4} and Figure \ref{densitydistribution} are both empirical plots for the case when $\mathfrak{m}_1$, $\mathfrak{m}_2$,$\mathfrak{m}_3$ are evenly spaced and each of them has arclength $\frac{7\pi}{12}$.  From \cite{Mc98}, the critical exponent in this case is $\delta=0.62627635$.  Figure \ref{dist4} is the plot for the gap distribution function $F_{T,\mathcal{I}}$, when $T=\sqrt{2}\times10^{3}$ and $\mathcal{I}$ is the full boundary $\partial{\DD}$.  Figure \ref{densitydistribution} illustrates the densities of the gap distributions, or ``$F_{T,\mathcal{I}}^{'}$'' for various $T$'s and $\mathcal{I}$'s (We didn't prove $F_{T,\mathcal{I}}^{'}$ exist everywhere exluding countably many points, although we expect this to be true.  Figure \ref{densitydistribution} are actually some normalized histograms with step taken to be 0.2). \\

The scale $1/T^2$ agree with what a random point process on $\Lambda_\Gamma$ predicts (See Appendix \ref{AppendixA}).  It is not completely unsurprising,  as $\mathcal{B}_T(1)$ tend to accumulate at those tiny clusters which make up the limit set $\Lambda_\Gamma$, instead of the whole interval.  An informal way to think of the scale $1/T^2$ is that $\mathcal{B}_T(1)$ are evenly spaced in a small neighborhood of $\Lambda_\Gamma$,  and since gaps are of scale $1/T^2$, the union of intervals containing $\mathcal{B}_T(1)$ form an efficient open cover of $\Lambda_\Gamma$; the number of such intervals is $T^{2\delta}$, which agree with the Hausdorff dimension of $\Lambda_\Gamma$.

 \begin{figure}[htbp] 
    \centering
    \includegraphics[width=3in]{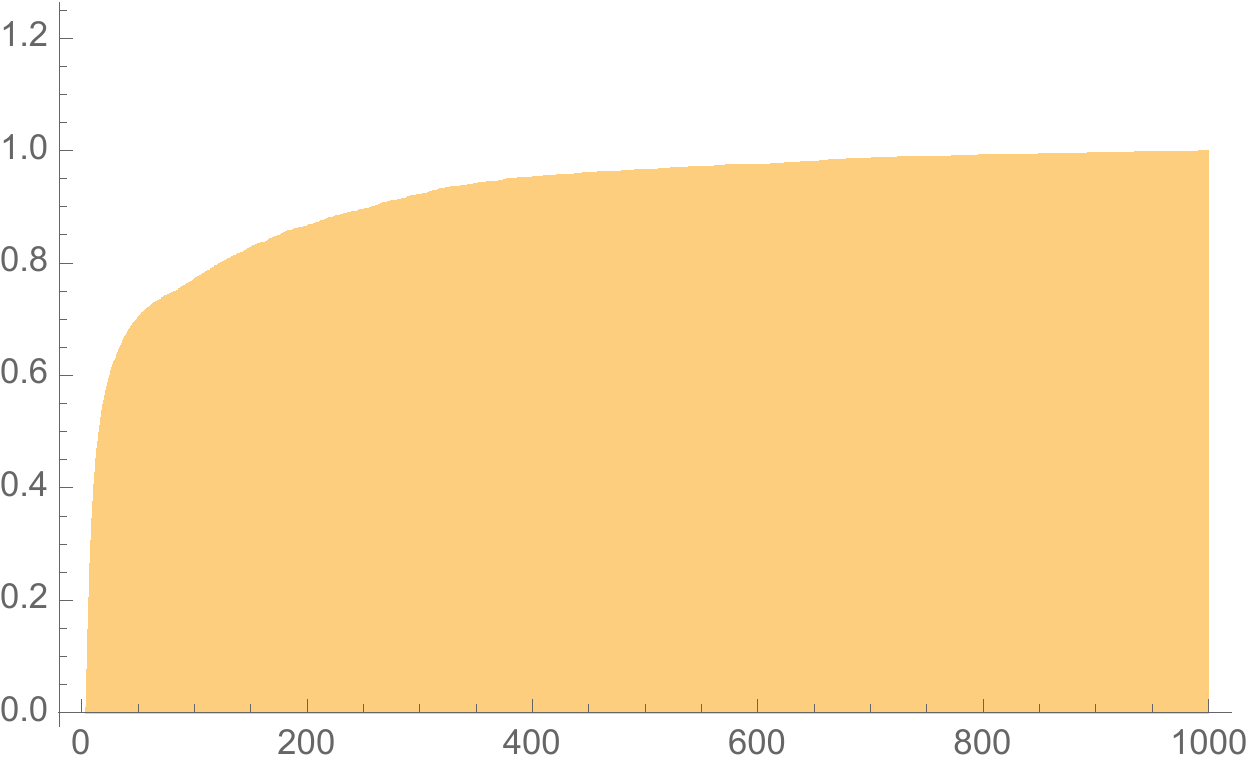} 
   \caption{The plot for the gap distribution function $F_{\sqrt{2}\times10^3,\partial{\DD}}$ }
     \label{dist4}
 \end{figure}
 
 \begin{figure}[h]
\centering
\begin{minipage}{.45\textwidth}
  \centering
  \includegraphics[width=2.4in]{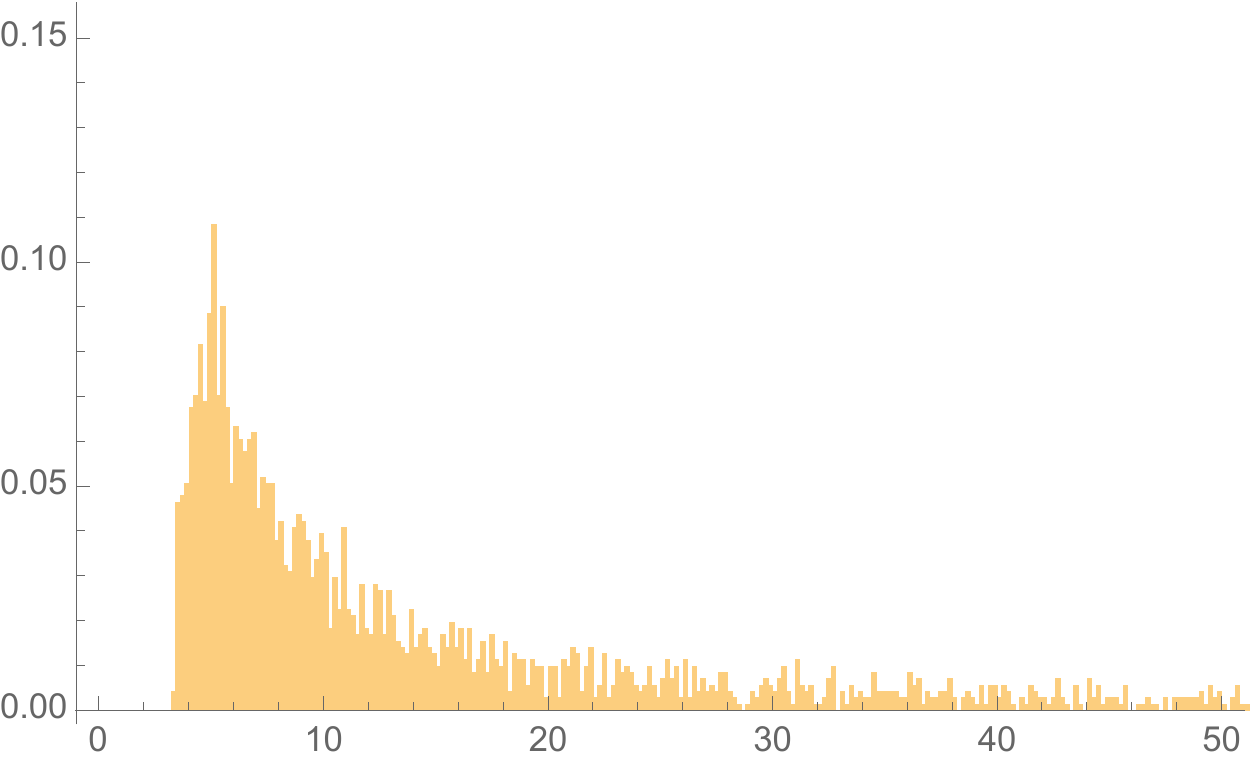}
  \subcaption{$T=\sqrt{2}\times10^3$, $\mathcal{I}=\partial{\DD}$}
  \end{minipage}
\begin{minipage}{.45\textwidth}
  \centering
  \includegraphics[width=2.4in]{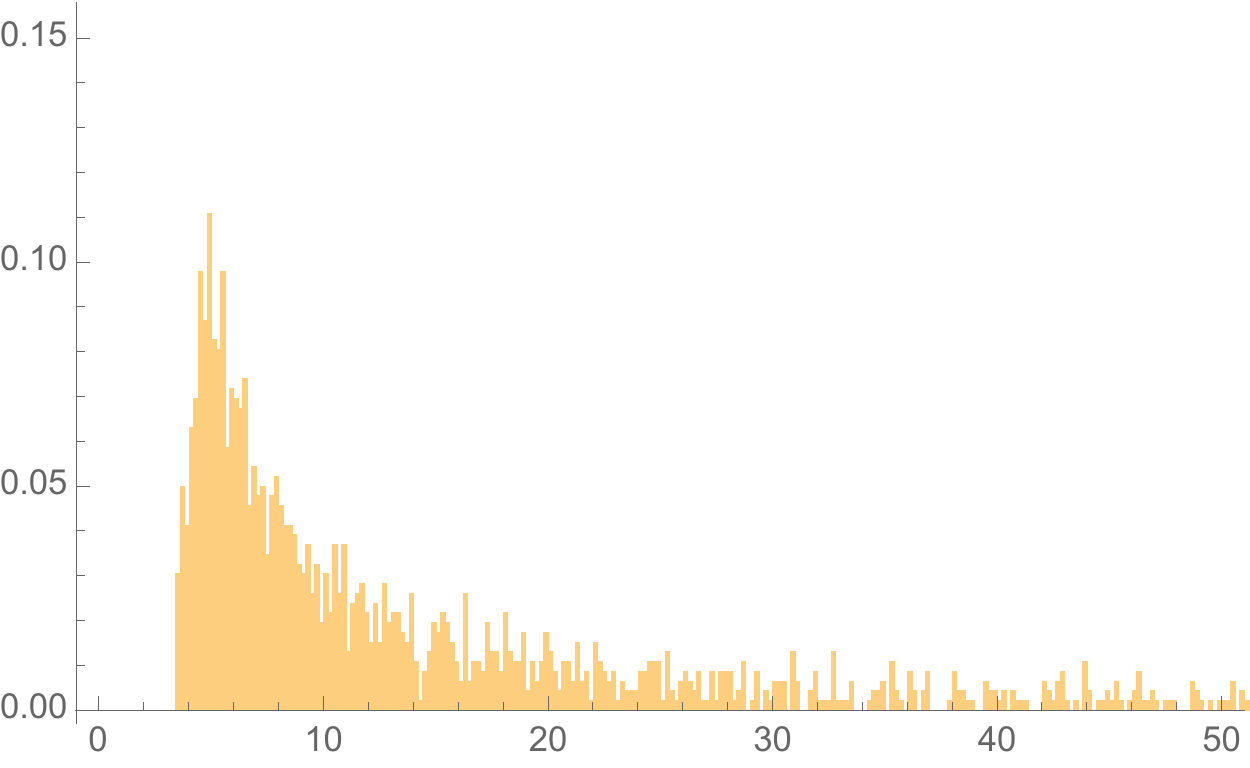}
   \subcaption{$T=10^3$, $\mathcal{I}=\partial{\DD}$}
  \end{minipage}\\
\begin{minipage}{.45\textwidth}
  \centering
  \includegraphics[width=2.4in]{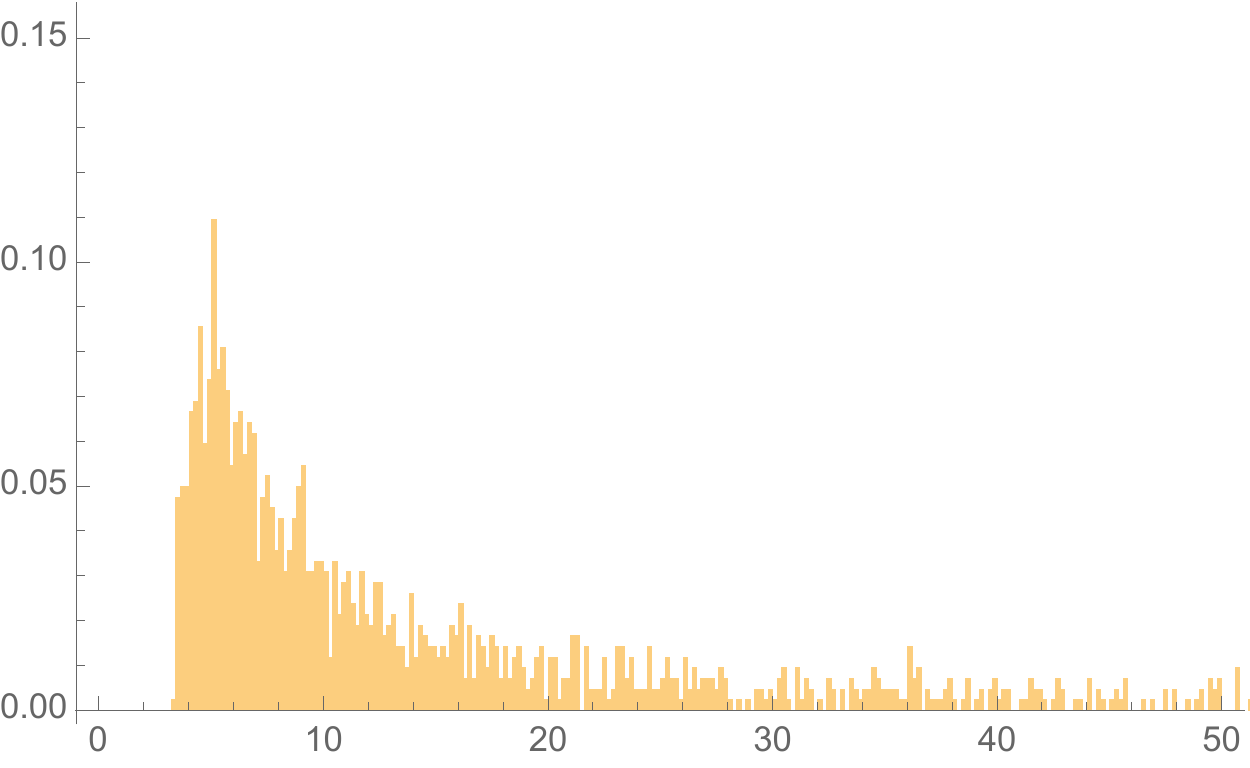}
   \subcaption{$T=\sqrt{2}\times10^3$, $\mathcal{I}=[0.695204,2.980334)$}
 \end{minipage}
\caption{The plot for $F^{'}_{T,\mathcal{I}}$ for various $T$'s and $\mathcal{I}$'s}
 \label{densitydistribution}
\end{figure}

The crucial tool that we use is a bisector counting theorem for $\Gamma$ (Theorem \ref{OS}).  The case when $\Gamma$ is a lattice, or equivalently $\delta=1$ was proved by Good \cite{Go83}.   Later it was generalized by Bourgain, Kontorovich and Sarnak to cover the cases when $\delta>\frac{1}{2}$ using the unitary representation theory of $SL(2,\RR)$ \cite{BKS10}.  Oh and Shah further extended this counting result  to cover the cases when $0<\delta\leq \frac{1}{2}$ using mixing of the geodesic flow for the Bowen-Margulis-Sullivan measure \cite{OS13}.\\

$\bd{Notation}$: For convenience we adopt the following notation throughout the paper:  Our main growing parameter is $T$. We write $f\sim g$ to mean $\frac{f}{g } \rightarrow 1$.  By $f\ll g $ $(f\gg g, f\asymp g)$ we mean there exists a positive constant $B$ depending only on $\Gamma$ such that $f\leq Bg $ $(f\geq Bg, \frac{1}{B}g\leq f \leq Bg)$.  We also use $f=O(g)$ which is synonymous with $f\ll g$, and $f=o(g)$ to mean $\frac{f}{g}\rightarrow 0$.  Without further explanation all the implied constants depend at most on $\Gamma$ and the norm $||\cdot||$.\\

We say a conditon A (about $T$) is \emph{asymptotically equivalent} to a condition B (also about $T$) if either of the two situations occur: (1) both 
$\#\{\gamma\in\Gamma \vert \gamma \text { satisfies A}\}$ and $ \#\{\gamma\in\Gamma \vert \gamma \text{ satisfies B}\}  $ are $o(T^{2\delta})$; (2) $$\lim_{T\rightarrow \infty}\frac{\#\{\gamma\in\Gamma \vert \gamma \text{ satisfies A}\}}{ \#\{\gamma\in\Gamma \vert \gamma \text{ satisfies B}\}  } =1.$$
 
$\bd{Plan\text{ } for\text{ } the\text{ } paper}$: In Sec. \ref{seccounting} we prove several counting theorems about $\Gamma$.  In Sec. \ref{seccompare} we prove several statements on the relationship between the norm $||\cdot||$ and the word length of an element when expressed as a word of $S_i$'s.  An auxiliary tool is to campare the norm $||\cdot||$ to the Frobenius norm $||\cdot||_{\text{Frob}}$ . In Sec. \ref{seccompare} we divide $\Gamma$ into infinite families according to the ending letters of $\gamma$, and the results in Sec. \ref{seccompare} guarantees that we only need to check bounded many cases when determining the neighbors of any element in $\mathcal{A}_{T,\mathcal{I}}$.  In Sec. \ref{secgap} we prove our main theorem, where the results in Sec. \ref{seccompare} allows us to pass the limit when assembling the contributions from these infinite families.  We conclude this paper by analyzing the limiting behavior of a random process on $\Lambda_\Gamma$ in the appendix.  \\
 
$\bd{Acknowledgement}$: The author would like to thank Elena Fuchs, Hee Oh, Jens Marklof and Ze\'ev Rudnick for useful discussions.

\section{Several Counting Statements \label{seccounting}}
Let $\DD$ be the Poincar\'e disc with the metric $$ds^2=\frac{4(dx^2+dy^2)}{(1-(x^2+y^2))^2}.$$  The orientation-preserving symmetry group of $\DD$ is 
$$G=PSU(1,1)=\left\{\mat{\xi&\eta\\\bar{\eta}&\bar{\xi}}\Big\vert |\xi|^2-|\eta|^2=1\right\} \cong PSL_2(\RR).$$

Let 
\aln{& K=\left\{\mat{e^{\frac{\phi\bd{i}}{2}}&\\&e^{\frac{-\phi\bd{i}}{2}}}\Big\vert \phi\in[0,2\pi)\right\},\\
& A=\left\{\mat{\cosh \frac{t}{2}& \sinh \frac{t}{2}\\\sinh\frac{t}{2}&\cosh\frac{t}{2} }\Big\vert t\in[0,\infty) \right\}.
}
Recall the Cartan decompositon $G=KA^{+}K$ that each $g\in G$ can be written in a unique way as 
$$g=k_{\phi_1(g)}a_{t(g)}k_{\pi-\phi_2(g)}=\mat{e^{\frac{\phi_1(g)}{2}\bd{i}}&\\&e^{\frac{-\phi_1(g)}{2}\bd{i}}}\mat{\cosh \frac{t(g)}{2}& \sinh \frac{t(g)}{2}\\\sinh\frac{t(g)}{2}&\cosh\frac{t(g)}{2}}\mat{e^{\frac{\pi-\phi_2(g)}{2}\bd{i}}&\\&e^{-\frac{\pi-\phi_2(g)}{2}\bd{i}}}$$
with $\phi_1(g),\phi_2(g)\in[0,2\pi)$ and $t(g)>0$, so this determines $\phi_1,\phi_2,t$ as functions of $g$.  The Haar measure is given by $dg=e^td\phi_1d\phi_2 dt$.\\

In this coordinate, $\gamma$ maps a circle $C((1-r)e^{\theta\bd{i}},r)$ to a circle of curvature 
\al{\label{curvatureformula}\nonumber\kappa(\gamma(C((1-r)e^{\theta\bd{i}},r)))&=e^{t(\gamma)}\frac{1-r}{2r}(1+\cos (\pi-\phi_2(\gamma)+\theta))\\&+e^{-t(\theta)}\frac{1-r}{2r}(1-\cos(\pi- \phi_2(\gamma)+\theta))+1}
and tangent to $\partial \DD$ at
 \al{\label{tranformula} exp\left({\left(\phi_1(\gamma)+\arcsin \frac{2\sin(\pi-\phi_2(\gamma)+\theta)}{e^{t(\gamma)}(1+\cos(\pi-\phi_2(\gamma)+\theta))+e^{-t(
\gamma)}(1-\cos(\pi-\phi_2(\gamma)+\theta))}\right)\bd{i} } \right).}\\

Our crucial tool is the following bisector counting theorem:  

\begin{thm}[Good, Bourgain-Kontorvich-Sarnak, Oh-Shah] \label{OS}Let $\Gamma$ be a non-elementary, geometrically finite discrete subgroup in $PSU(1,1)$.  Let $\mathcal{I},\mathcal{J}$ be two intervals in $[0,2\pi)$ with $\mu(\mathcal{I}),\mu(\mathcal{I})>0$, then 
$$\sum_{\gamma\in\Gamma}\bd{1}
  \left\{
  \begin{array}{@{}ll@{}}
    \phi_1(\gamma)\in\mathcal{I}\\
    \phi_2(\gamma)\in\mathcal{J}\\
    t(\gamma)\leq T
  \end{array}\right.\sim c_{\Gamma}\mu(\mathcal{I})\mu(\mathcal{J})e^{\delta T}
$$
for some positive constant $c_{\Gamma}$ which only depends on $\Gamma$.
\end{thm}

Theorem \ref{OS} has the following corollary which is applied to our analysis:
\begin{cor}\label{11081155}Let $\mathcal{I}$ be as above. Define a measure $dS=e^{\delta t}dt\times d\mu$ on $[0,\infty)\times[0,2\pi)$. Let $\Omega=\Omega_0$ be a bounded set with piecewise smooth boundary in $[0,\infty)\times[0,2\pi)$ and $S(\Omega_0)>0$, and let $\Omega_T=\{(t,\theta)\in[0,\infty)\times[0,2\pi)\big\vert ({t}-{T},x)\in\Omega\}$.  Then as $T$ goes to $\infty$,
\aln{\sum_{\gamma\in\Gamma} \left\{
  \begin{array}{@{}ll@{}}
    \phi_1(\gamma)\in\mathcal{I}\\
    (t(\gamma),\phi_2(\gamma))\in\Omega_T
  \end{array}\right.\sim c_\Gamma\mu(\mathcal{I})S(\Omega_0)e^{\delta T}.}
   \end{cor}
   \begin{proof}
   This follows line by line from the argument in Proposition 5.2 of \cite{RZ15}, with the extra ingredient that $\mu$ is atomless, so we can approximate $\Omega_T$ by sectors with the errors under control.  
   \end{proof}
   
Each $\gamma\in\Gamma_0$ can be expressed in a unique way as $S_1S_2\cdots S_{l(\gamma)}$, where each $S_i\in \{\rho_1,\rho_2,\rho_3\}$ and $S_{i}\neq  S_{i+1}$.  We call $S_1S_2\cdots S_{l(\gamma)}$ is the reduced word for $\gamma$ and we say the word length of $\gamma$ is $l(\gamma)$.  We denote the initial letter $S_1$ for $\gamma$ by $\rm{Int}(\gamma)$, and the ending letter $S_{l(\gamma)}$ for $\gamma$ by $\rm{End}(\gamma)$.    Without further mentioning, all the words in this paper are reduced.  \\

Now we can give an asymptotic count for $\# \mathcal{A}_{T,\mathcal{I}}$:
\begin{thm}\label{countpoints} There exists a positive constant $c_0$ such that, as $T$ goes to $\infty$, 
$$\# \mathcal{A}_{T,\mathcal{I}}\sim c_0\mu(\mathcal{I})T^{2\delta}.$$
\end{thm}
\begin{proof}
We split $\mathcal{A}_{T,\mathcal{I}}$ into two sets as we are working with $\Gamma$ instead of $\Gamma_0$:
\al{\label{11121035}\mathcal{A}_{T,\mathcal{I}}=\left\{\gamma\in\Gamma\Big\vert ||\gamma||<T,\gamma(1)\in\mathcal{I}\right\}\cup \left\{\gamma\in\Gamma\Big\vert ||\gamma S_1||<T, \rm{End}(\gamma)\neq S_1, \gamma S_1(1)\in\mathcal{I}\right\}}
For the first set in the above union, since $1$ is not a limit point, $e^{\phi_1(\gamma)\bd{i}}, e^{\phi_2(\gamma)\bd{i}}$ stays uniformly away from 1, or $\phi_1(\gamma),\phi_2(\gamma)$ stays uniformly away from 0 (more precisely there is an open neighborhood of 1 in $\partial\DD$ such that $e^{\phi_1(\gamma)\bd{i}},e^{\phi_2(\gamma)\bd{i}}$ stays away from this neighborhood with only finitely many exceptions).  From \eqref{curvatureformula}, the condition $$||\gamma||<T$$ is asymptotically equivalent to  
\al{e^{t(\gamma)}\frac{1-r_0}{2r_0}(1+\cos(\pi-\phi_2(\gamma)))< T^2.
}
From \eqref{tranformula}, the condition $$\gamma(1)\in\mathcal{I}$$ is asymptotically equivalent to 
\al{\phi_1(\gamma)\in\mathcal{I}.}
Set \aln{\Omega_T^0=\left\{(t,\theta)\Big\vert e^{t}\frac{1-r_0}{2r_0}(1+\cos(\pi-\phi_2(\gamma)))\leq T, \phi_1(\gamma)\in\mathcal{I}   \right   \},}  then 
\aln{ \left\{\gamma\in\Gamma\Big\vert ||\gamma||<T,\gamma(1)\in\mathcal{I}\right\} \sim \left\{\gamma\in\Gamma \Big\vert \big(t(\gamma),s(\gamma)\big)\in \Omega_{T^2}^0 \right \}  }
Clearly $\Omega_{T}^0=\log T+\Omega_0^0$.  Therefore, applying Corollary \ref{11081155} to $\Omega_{T^2}^0$, we obtain 
$$\left\{\gamma\in\Gamma\Big\vert ||\gamma||<T,\gamma(1)\in\mathcal{I}\right\}\sim c_0^{(1)}\mu(\mathcal{I})T^{2\delta}$$ 
with $c_{0}^{1}=c_{\Gamma}S(\Omega_0^0)$.\\

The second set in the union \eqref{11121035} has an extra condition 
$$\text{End}(\gamma)\neq S_1,$$
which is asymptotically equivalent to 
$$\phi_2(\gamma)\not\in\mathfrak{m}_1,$$
or $$\phi_2(\gamma)\in \mathfrak{m}_2 \cup \mathfrak{m}_3.$$
By the same reasoning as for the first set, we have 
\aln{\left\{\gamma\in\Gamma\Big\vert ||\gamma S_1||<T, \rm{End}(\gamma)\neq S_1, \gamma S_1(1)\in\mathcal{I}\right\}\sim c_0^{(2)}\mu(\mathcal{I})T^{2\delta}}
for some $c_0^{(2)}>0$.
Thus Theorem \ref{countpoints} follows, once we set $c_0=c_0^{(1)}+c_0^{(2)}$.
\end{proof}

 \section{Compare Norms\label{seccompare}}
 In this section we prove several relationships between $||\cdot||$ and $l(\cdot)$.
 
 \begin{lem}\label{equivalentnorms}The norms $||\cdot||$ and $||\cdot||_{\rm{Frob}}$ are equivalent in $\Gamma$.
 \end{lem}
 \begin{proof}
 It's easy to see this from the Cartan Decomposition.  Write $\gamma=k_{\phi_1(\gamma)}a_{t(\gamma)}k_{\phi_2(\gamma)}$. Then $||\gamma||_{\rm{Frob}}^2=||a_{t(\gamma)}||^2=\frac {e^t+e^{-t}}{2}\asymp e^t$, since $||\cdot||_{\text{Frob}}$ is invariant under $K$.  From \eqref{curvatureformula} we have
$$||\gamma||^2=\kappa(\gamma(C))=e^{t}\frac{1-r_0}{2r_0}(1+\cos (\pi-\phi_2(\gamma))+e^{-t}\frac{1-r_0}{2r_0}(1-\cos(\pi- \phi_2(\gamma)))+1.$$
Therefore, we also have $||\gamma||^2\asymp e^t$ since when $t$ large, because $\phi_2(t)$ stays away from $0$ as $1\not\in \Lambda_\Gamma$.
\end{proof}

 \begin{figure}[htbp] 
    \centering
    \includegraphics[width=2.1in]{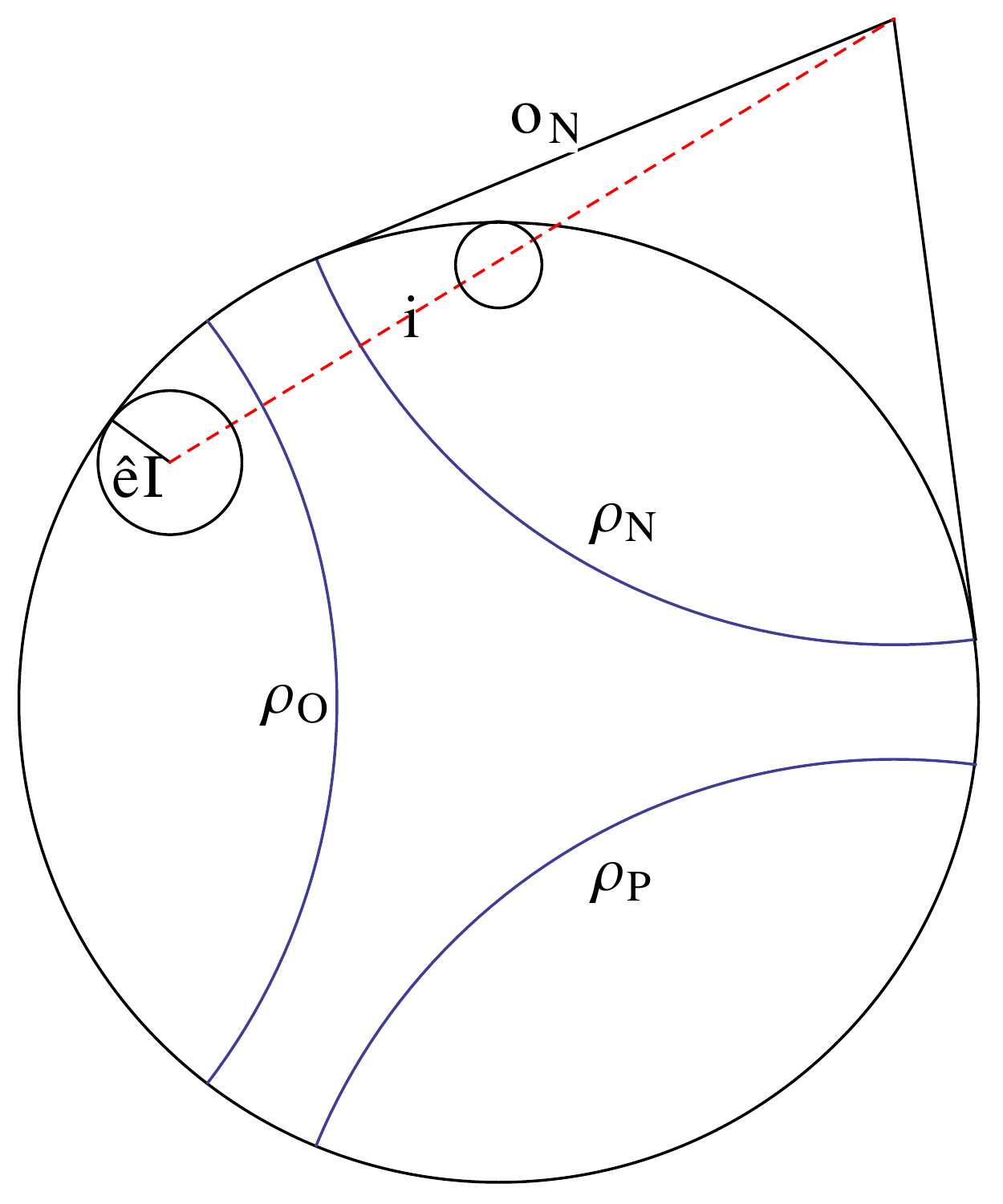} 
   \caption{Reflecting a circle}
    \label{fig1}
 \end{figure}

 Recall $\gamma=S_1S_2\cdots S_{l(\gamma)}$ is the reduced word for $\gamma$.  The following lemma gives a relationship between $||\gamma||$ and $l(\gamma)$. 
 \begin{lem}\label{11021243}There exists $b>a>1$, such that 
 $$a^{l(\gamma)} \ll ||\gamma||\ll b^{l(\gamma)}.$$
 \end{lem}
 \begin{proof}
 
The isometry circles of $\rho_1,\rho_2,\rho_3$ are $C_1,C_2,C_3$, and each of them divides $\partial \DD$ into a major arc $\mathfrak{M}_1,\mathfrak{M}_2,$ or $\mathfrak{M_3}$ and a minor arc $\mathfrak{m}_1,\mathfrak{m}_2,$ or $\mathfrak{m}_3$.   By induction one can show that for each $i$, $S_i\cdots S_n(1)\in\mathfrak{m}_i\subset \mathfrak{M}_{i-1}$. Each $\rho_i$  maps a circle tangent on $\mathfrak{M}_i$ to a circle tangent on $\mathfrak{m}_i$ with smaller radius.  The exact relation is illustrated as follows: 
 
 Let $R_i$ be the radius of $C_i$. Let $C^{'}$ be any circle sitting on $\mathfrak{m}_j(j\neq i)$ with radius $r^{'}<r_0$ so that $C^{'}$ does not intersect $C_i$ (See Figure \ref{fig1} where we set $i=1$ and $j=2$).  Let $L$ be the distance between the centers of $C_i$ and $C^{'}$.  Then $\rho_i$ will map $C^{'}$ tangent to $\mathfrak{m}_j$ with curvature
 \al{\label{10310847}\frac{L^2-{r^{'}}^2}{R_i^2} \frac{1}{r^{'}}.}

 We have the following crude estimate
 $$\frac{(R_i+2)^2}{R_i^2}>\frac{l^2-{r^{'}}^2}{R_i^2}>1+\frac{r_0}{3R_i}.$$
 Let $a=\min_i\{1+\frac{r}{3R_i}\}$ and $b=\max_i\{\frac{(R_i+2)^2}{R_i^2}\}$.  Then by induction we can prove that 
 $$a^n\frac{1}{r_0}< \kappa(\gamma(C)) < b^n\frac{1}{r_0}.$$
 \end{proof}

  The above lemma has the following generalization:
  \begin{lem}\label{11081149}Let $w$ be any reduced word of length $n$.  If $\text{End}(\gamma)\neq\text{Int}(w)$, then 
  $$a^n||\gamma||\ll ||w\cdot\gamma||\ll b^n||\gamma||.$$
  Similarly,  if $\text{Int}(\gamma)\neq\text{End}(w)$, then 
  $$a^n||\gamma||\ll ||\gamma\cdot w||\ll b^n||\gamma||.$$
  \end{lem}
  \begin{proof}
  If $\rm{Int}(\gamma)=\rho_i$ and $\text{End}(w)=\rho_j$ with $j\neq i$, then $\gamma(C_I)$ is a circle sitting on $\mathfrak{m}_i$.  Then a same argument as in Lemma \ref{11021243} shows that 
  $$a^n \kappa(\gamma(C)) \ll \kappa(w(\gamma(C)))\ll b^n \kappa(\gamma(C)),$$
  or $$a^n||\gamma||\ll ||w\cdot\gamma||\ll b^n||\gamma||.$$
  The second case is reduced to the first case immediately, using the fact that $||\gamma \cdot w||\asymp ||\gamma \cdot w||_{\rm{Frob}}$ and $||\gamma \cdot w||_{\rm{Frob}}=||(\gamma \cdot w)^{-1}||_{\rm{Frob}}$.
  \end{proof}
  We will also need the following lemma:
\begin{lem}\label{02141541}Let $\rho_i,\rho_j,\rho_k$ be three different letter from $\{\rho_1,\rho_2,\rho_3\}$ and $\gamma\in\Gamma_0$ be any word with $\text{End}(\gamma)=\rho_i$, then we have 
$$||\gamma(\rho_j\rho_k)^n ||\asymp ||\gamma(\rho_k\rho_j)^n ||,$$
and the implied constants in independent of $n$.
\end{lem}

\begin{proof}
We first assume $\gamma\in \Gamma$.  There exists a constant $N_0$ which depends only on $\Gamma$ such that whenever $l(\gamma)>N_0$, the condition $\rm{End}(\gamma)=\rho_i$ is equivalent to the condition $\phi_1(\gamma)\in\alpha_i$.\\

We can exclude the cases when both $l(\gamma)\leq N_0$ and $n\leq\frac{N_0}{2}$ since there are only finitely many such $\gamma$'s and $n$'s.  We just deal with the case $l(\gamma)>N_0$; the case $n>\frac{N_0}{2}$ is similar.  \\

Since $\left((\rho_j\rho_k)^n\right)^{-1}=(\rho_k\rho_j)^n$, we have $\kappa((\rho_j\rho_k)^n(C_{\mathcal{I}}))\asymp \kappa((\rho_k\rho_j)^n(C_{\mathcal{I}}))$.  Since $(\rho_j\rho_k)^n(C)$ has tangencies $t_1,t_2$ in $\alpha_j,\alpha_k$, and $\phi_2(\gamma)\in\alpha_i$, we have both $\phi_2(\gamma)+\pi-t_1 $ and $\phi_2(\gamma)+\pi-t_2 $ are bounded away from $\pi$.  We can apply \eqref{curvatureformula} to  
$(\rho_j\rho_k)^n(C_{\mathcal{I}})$ and $(\rho_k\rho_j)^n(C_{\mathcal{I}})$ to see that 
$$\kappa(\gamma(\rho_j\rho_k)^n(C_{\mathcal{I}}))\asymp \kappa(\gamma(\rho_k\rho_j)^n(C_{\mathcal{I}})),$$
or $$||\gamma(\rho_j\rho_k)^n||\asymp ||\gamma(\rho_k\rho_j)^n ||.$$

If $\gamma\not\in\Gamma$, then write $\gamma=\rho\gamma^{'}$ where $\rho\in\{\rho_1,\rho_2,\rho_3\}$, $\rho\neq \text{Int}(\gamma^{'})$ and $\gamma^{'}\in\Gamma$.  Then $\kappa\left(\gamma^{'}(\rho_i\rho_j)^n(C_{\mathcal{I}})\right) \asymp \kappa\left(\gamma^{'}(\rho_j\rho_i)^n(C_{\mathcal{I}})\right)$.  Using the argument in Lemma \ref{11021243} we see that 
$$a\cdot\kappa\left(\gamma^{'}(\rho_i\rho_j)^n(C_{\mathcal{I}})\right)< \kappa\left(\gamma(\rho_i\rho_j)^n(C_{\mathcal{I}})\right)<b\cdot\kappa\left(\gamma^{'}(\rho_i\rho_j)^n(C_{\mathcal{I}})\right)$$
and
$$a\cdot\kappa\left(\gamma^{'}(\rho_j\rho_i)^n(C_{\mathcal{I}})\right)< \kappa\left(\gamma(\rho_j\rho_i)^n(C_{\mathcal{I}})\right)<b\cdot\kappa\left(\gamma^{'}(\rho_j\rho_i)^n(C_{\mathcal{I}})\right).$$
Therefore, we also have 
$$||\gamma(\rho_j\rho_k)^n||\asymp ||\gamma(\rho_k\rho_j)^n ||$$
in this case. 
\end{proof}

\section{Parametrization of gaps\label{secparametrization}}

Each gap formed by $\mathcal{A}_{T,\mathcal{I}}$ has two end points $\gamma_1(1),\gamma_2(1)$ with $||\gamma_1||,||\gamma_2||<T$.  For each $\gamma_1(1)$, we want to determine its neighbors.  Figure \ref{fig2} illustrates that circles tend to accumulate at some ``dense" parts while are sparse at other parts (see Figure \ref{densesparse}).  If $\gamma_1(C_I)$ lie in a dense part, typically the tail of the word of $\gamma_1$ contains a string $\rho_1\rho_3\rho_1\rho_3\cdots$; if $\gamma_1(C_I)$ lie in a sparse part, typically the tail of $\gamma_1$ contains a string $\rho_1\rho_2\rho_1\rho_2\cdots$ or $\rho_2\rho_3\rho_2\rho_3\cdots$.  The neighbor determination is different in these two cases.  This suggests us to divide $\Gamma_0$ according to the ending of the words.  \\

First we choose an integer $N_1$ that satisfies the following conditions:
\begin{enumerate}
\item For any $\gamma\in\Gamma_0$ with $\rm{End}(\gamma)=\rho_i$, 
$||\gamma\rho_j||\leq ||\gamma \rho_k w||$ for any reduced word $w$ 
with $l(w)\geq N_1$ and $\text{Int}(w)\neq \rho_k$.
\item For  any $\gamma\in\Gamma_0$ with $\rm{End}(\gamma)=\rho_i$, 
$||\gamma(\rho_j\rho_k)^n|| \leq \gamma(\rho_k\rho_j)^n w$ for any reduced 
word $w$ with $l(w)\geq N_1$ and $\text{Int}(w)\neq \rho_j$.
\end{enumerate}

Let $\mathcal{W}_{N_1}$ be the collection of $w\in\Gamma_0$ with $l(w)=N_1$ or $N_1+1$. We drop out those elements $\gamma_1$ in $\Gamma_0$ with $l(\gamma_1)\leq N_1$, and those $\gamma_1$ of the form $\gamma_1=(\rho_i\rho_k)^nw$ or $\rho_k(\rho_i\rho_k)^nw$ and call such elements \emph{sporadic} (there are only $\ll \log T$ such elements in $\mathcal{B}_T$ so this doesn't affect the limiting gap distribution).  Then we can divide the rest terms in $\Gamma_0$ into disjoint classes $\mathcal{C}_{i,j,k}^n$, where again $\rho_i,\rho_j,\rho_k$ are three different letters from $\{\rho_1,\rho_2,\rho_3\}$:

$$\mathcal{C}_{i,j,k}^n=\{\gamma_1=\gamma (\rho_j\rho_k)^nw\vert \gamma\in \Gamma, w\in\mathcal{W}_{N_1}, \text{End}(\gamma)= \rho_i, \text{Int} (w)\neq \rho_k \}.$$
Our goal is to determine $\gamma_2(1)$ for each $\gamma_1(1)$ in $\mathcal{A}_{T,\mathcal{\partial \DD}}$.
Each point in $\gamma\rho_i(1)\in\mathcal{C}_{i,j,k}^{n}$ has two sides.  For convenience, we assign $-$ to the side where $\gamma\rho_i(\rho_j\rho_k)^{n-1}\rho_j\rho_i(1)$ lies in, and $+$ to the other side.  And we are going to give asymptotically equivalent conditions that determine the neighbor from each side.\\

Write $\gamma=S_1S_2\cdots S_{l(\gamma)}$ in reduced form.  Since $S_{l(\gamma)}$ is from $\{\rho_1,\rho_2,\rho_3\}$, it corresponds to a minor arc $\mathfrak{m}$ from $\{\mathfrak{m}_1,\mathfrak{m}_2,\mathfrak{m}_3\}$.  We then define $\mathfrak{m}_{\gamma}$ to be the arc $S_1S_2\cdots S_{l(\gamma)-1}(\mathfrak{m})$.  We have the following simple observation:
 \begin{lem}\label{immediate} If $\gamma  (1)$ lies on one side of $\gamma \beta (1)$ in $\mathcal{A}_T$, then the neighbor of $\gamma \beta (1)$ on this side is of the form $\gamma \beta^{'} (1)$.
 \end{lem}
 \begin{proof} Both $\gamma (1)$ and $\gamma\beta(1)$ lie in the arc $\alpha_{\gamma}$ and $||\gamma||<||\gamma\beta||$, so the neighbor of $\gamma \beta(1)$ on the side of $\gamma(1)$ also lies in the arc $\alpha_{\gamma}$.  Therefore, this neighbor is of the form $\gamma \beta^{'}(1)$.
 \end{proof}
 \begin{figure}
\centering
\begin{minipage}{.45\textwidth}
  \centering
  \includegraphics[width=.7\linewidth]{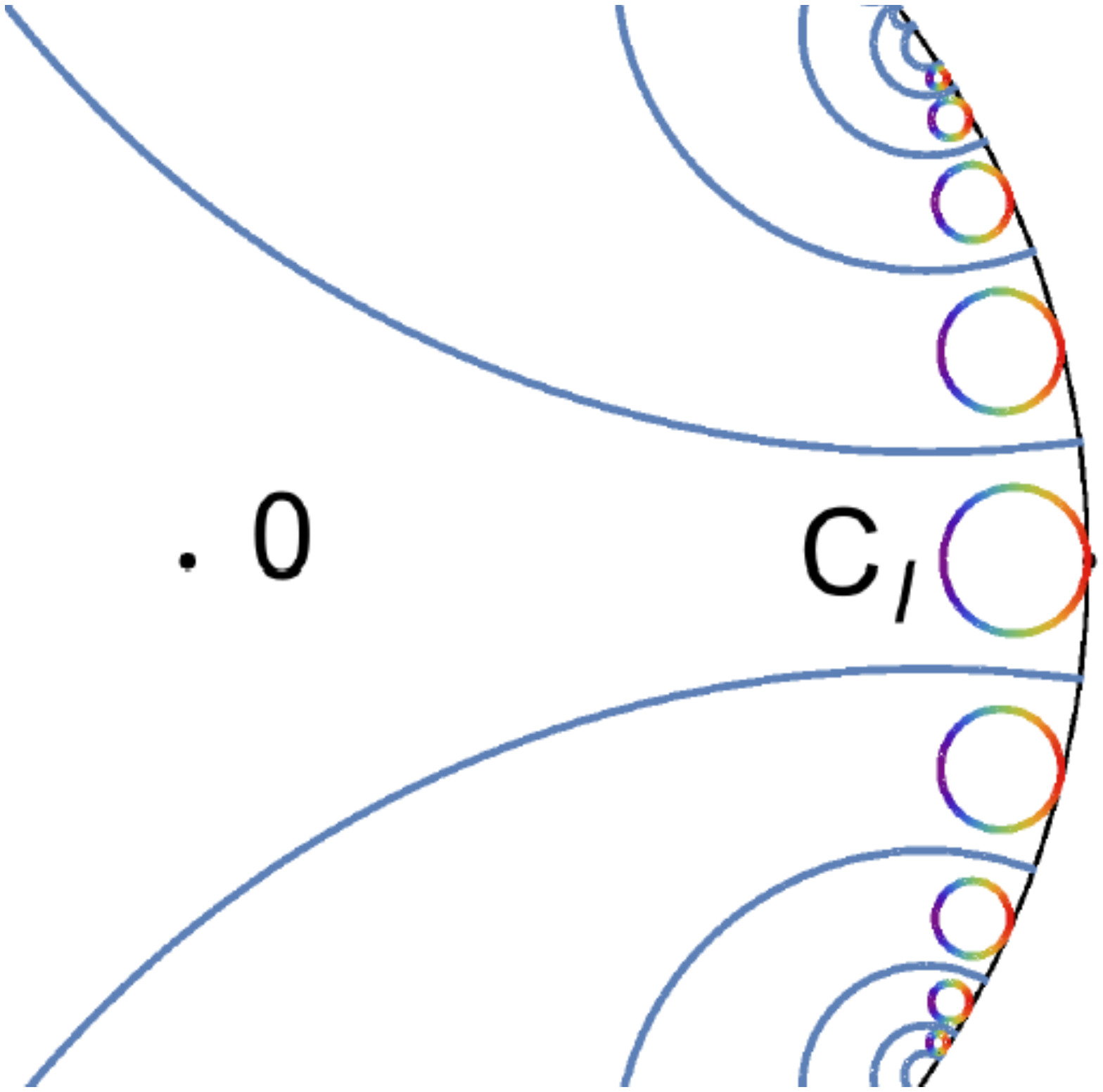}

\end{minipage}%
\begin{minipage}{.5\textwidth}
  \centering
  \includegraphics[width=.7\linewidth]{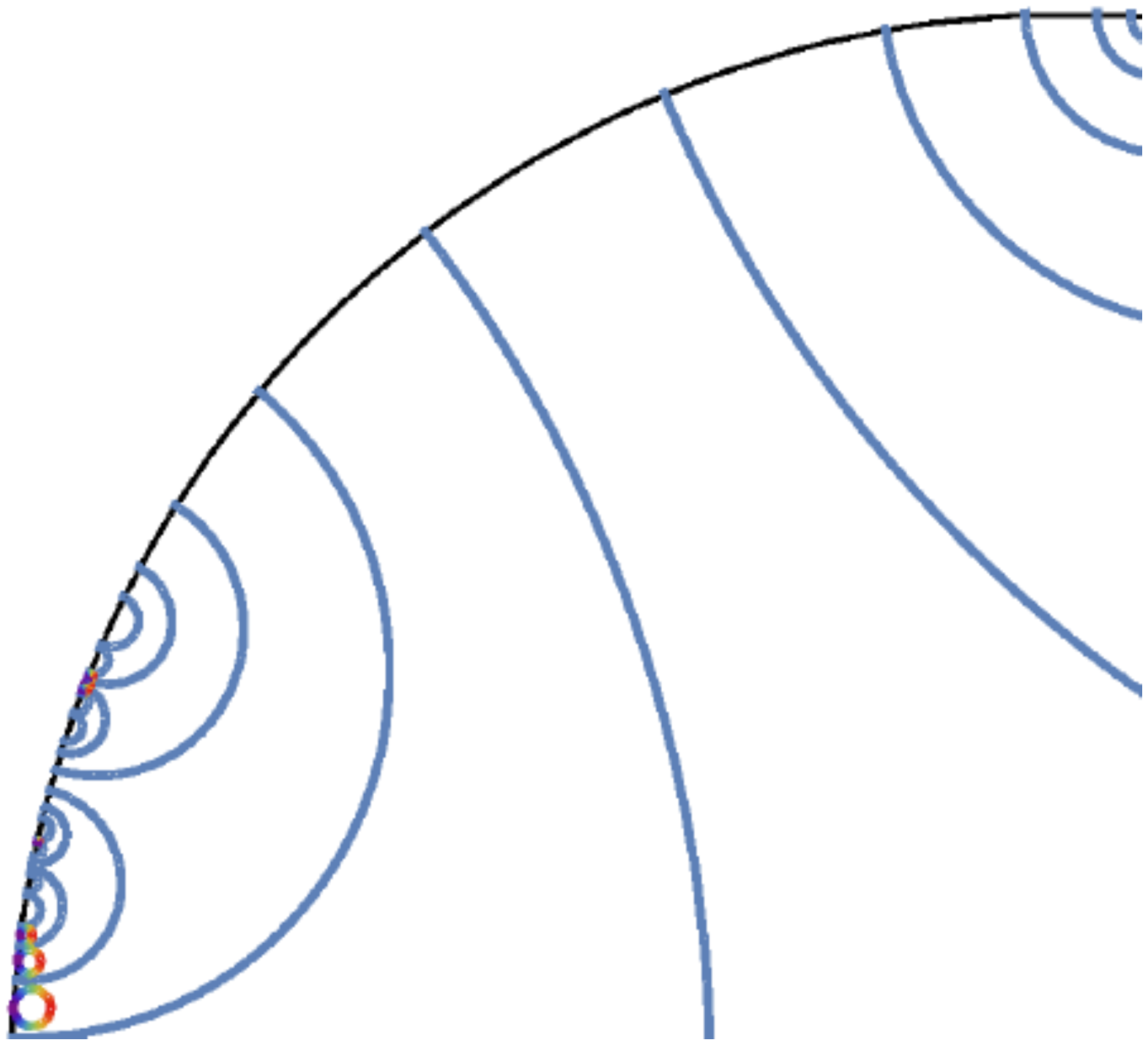}

\end{minipage}
 \caption{A dense part and a sparse part}
 \label{densesparse}
\end{figure}

This lemma helps us to determine what the dominated neighbor of $\gamma 1$ should be like.  Figure \ref{fig2} illustrates that circles tend to accumulate at some parts, while at other parts the circles are sparse (see Figure \ref{densesparse}).  This exactly corresponds to two typical cases.  We only analyze one example from each case.  \\

If a tangency is of the form $\gamma (\rho_1\rho_2)^n w(1)$ with $\text{End}(\gamma)=\rho_3, w\in\mathcal{W}_{N_1}$ and $\text{Int}(w)\neq \rho_2$, which is a representative for the first case, then on the $+$ side lies $\gamma (\rho_1\rho_2)^n(1)$.  Since $||\gamma (\rho_1\rho_2)^n||<||\gamma (\rho_1\rho_2)^nw||$,  from Lemma \ref{immediate}, the neighbor $\gamma_2(1)$ on this side is of the form $\gamma(\rho_1\rho_2)^{n} w^{'}(1)$ with $\text{Int}(w^{'})\neq \rho_2$.  On the $-$ direction lies $\gamma (\rho_1\rho_2)^{n-1}\rho_1\rho_3 (1)$ and $||\gamma (\rho_1\rho_2)^{n-1}\rho_1\rho_3||<||\gamma (\rho_1\rho_2)^nw||$.  So Lemma \ref{immediate} tells us that the neighborhood on the other side of $\gamma(\rho_1\rho_2)^n (1)$ is of the form $\gamma(\rho_1\rho_2)^{n-1}\rho_1\rho_3 w^{'}(1)$ with $\text{Int}(w^{'})\neq \rho_3$.   \\
 
A representative for the second case is that $\gamma_1=\gamma(\rho_2\rho_3)^n w$ with $\text{End}(\gamma)=\rho_1, w\in\mathcal{W}_{N_1}$ and $\text{Int}(w)\neq \rho_3$.  The neighbor on the $-$ side lies either in $\mathfrak{m}_{\gamma (\rho_2\rho_3)^{n}}$ or $\mathfrak{m}_{\gamma (\rho_2\rho_3)^{n-1}\rho_2\rho_1w^{'}}$, so the neighbor on this side is of the form $\gamma (\rho_2\rho_3)^{n}w^{'}$ or $\gamma (\rho_2\rho_3)^{n-1}\rho_2\rho_1w^{'}$.  On the $+$ side the neighbor must lie in either  $\mathfrak{m}_{\gamma (\rho_2\rho_3)^{n}w^{'}}$ or $\mathfrak{m}_{\gamma(\rho_3\rho_2)^{n}w^{'}}$, because   $\mathfrak{m}_{\gamma (\rho_2\rho_3)^{n}w^{'}}$ or $\mathfrak{m}_{\gamma.(\rho_3\rho_2)^{n}w^{'}}$ are the closest arcs on this side to $\gamma_1(1)$ that contains orbit points, and  $||\gamma (\rho_2\rho_3)^{n}||,||\gamma(\rho_3\rho_2)^{n}||\leq ||\gamma_1||$ from Lemma \ref{11081149} and Lemma \ref{02141541}.  Thus this neighbor must be of the form $\gamma (\rho_2\rho_3)^{n}w^{'}(1)$ or $\gamma(\rho_3\rho_2)^{n}w^{'}(1)$.  \\

In fact, the above argument also show that in each case, if $\gamma_2(1)$ is a neighbor of $\gamma_1(1)$, then we must have $||\gamma_2||\ll_{N_1}||\gamma_1||$, using Lemma \ref{11081149} and Lemma \ref{02141541}.  Since being neighbors is a mutual relation, we also have $||\gamma_2||\gg_{N_1}||\gamma_1||$.  Thus we have the following theorem:

\begin{thm}\label{comparable}If $\gamma_1(1)$ and $\gamma_2(1)$ are neighbors in $\mathcal{A_{T,\partial\DD}}$, then $||\gamma_1||\asymp ||\gamma_2||$, and the implied constant depends only on the group $\Gamma$.
\end{thm}

Theorem \ref{comparable}, together with  Lemma \ref{11081149} and Lemma \ref{02141541} also restricts the length of $w^{'}$ in each case.  We know there exists a positive constant $N_2$ which only depends on $N_1$, thus only depends on the group $\gamma$, such that for any $w^{'}$ appearing in any of the above settings, we have $l(w^{'})\leq N_2$.   \\

For each $\gamma_1\in\Omega_{i,j,k}^n \cdot w$ with $w\in\mathcal{W}_{N_1}$ and $\text{Int}{(w)}\neq \rho_k$ (i.e. $ \gamma_1=\gamma(\rho_j\rho_k)^n w$ with $\text{End}(\gamma_1)=\rho_i$), and each direction $\square\in\{+,-\}$ we order all possible choices for $w^{'}$ from closest to furthest to $\gamma_1(1)$ in this direction:  $u_1^{\square}, u_2^{\square},\cdots,u_{N_{w}^{\square}}^{\square}$.  It should be noticed that this order depends only on $i,j,k$ and is independent of $\gamma$ and $n$ because $\Gamma_0$ is generated by $\rho_1,\rho_2,\rho_3$, each of which is a diffeomorphism on $\partial\DD$ and preserves order.  For our later notational convenience we let $v=  (\rho_j\rho_k)^n \cdot w$ and $v_m^{\square}=(\rho_j\rho_k)^n \cdot u_m^{\square}$.

 \section{Analyzing the gap distribution function \label{secgap}} 
We let $\tilde{\Omega}_{i,j,k}^{n}(w,\square,l,s; T)$ be the set of $\gamma\in\Gamma$ such that $\rm{End}(\gamma)=\rho_i$, $\gamma (\rho_j\rho_k)^{n}w(1)$, $\gamma (\rho_j\rho_k)^{n}w_{l^{\square}}(1)$ are two end points of a gap in $\mathcal{A}_{T,\mathcal{I}}$ with relative length $\leq s$.  Here $w\in\mathcal{W}_{N_1}$ with $\text{Int}(w)\neq \rho_k$, and $\square\in\{+,-\}$.\\

The condition $\gamma\in\Omega_{i,j,k}^{n}(w,\square,l,s; T)$ is equivalent to the following conditions:
\al{\label{02141724}&\gamma v(1)\in\mathcal{I}\\
\label{110814470}&\rm{End}(\gamma)= \rho_i,\\
\label{110814471}&||\gamma v ||< T^2,\\
\label{110814472}&||\gamma v_m^{\square}||\geq T^2|\text{ for $1\leq m< l$},\\
\label{110814473}&  ||\gamma v_{l^{\square}}||<T^2,\\
\label{110814474}& \frac{d(\gamma v,\gamma v_{l}^{{\square}})}{T^2}\leq s.
}

The condition $\gamma v(1)\in\mathcal{I}$ is asymptotically equivalent to the condition 
\al{\label{02141724}\phi_2(\gamma)\in \alpha_{i}. } 

The condition $\rm{End}(\gamma)= \rho_i$ is asymptotically equivalent to the condition 
\al{\label{11082009}\phi_2(\gamma)\in \alpha_{i}. } 

Recall the notation $C_{\gamma}=\gamma (C_\mathcal{I})=C(1-r_\gamma,\theta_\gamma)$.  From \eqref{curvatureformula} we have
\aln{||\gamma v_m^{\square} ||&=e^{t(\gamma)}\frac{1-r_{v_m^{\square}}}{2r_{v_m^{\square}}}(1+\cos(\pi-\phi_2{(\gamma)}+\theta_{v_m}))+e^{-t(\gamma)}\frac{1-r_{v_m^{\square}}}{2r_{v_m^{\square}}}(1-\cos(\pi-\phi_2{(\gamma)}+\theta_{v_m^{\square}}))+1 \\
&\sim e^{t(\gamma)}\frac{1-r_{v_m^{\square}}}{2r_{v_m^{\square}}}(1+\cos(\pi-\phi_2(\gamma)+\theta_{v_m^{\square}})).
}

Therefore, conditions \eqref{110814471}, \eqref{110814472}, \eqref{110814473} are asymptotically equivalent to 
\al{\label{11082006}&e^{t(\gamma)}\frac{1-r_{v}}{2r_{_{v}}}(1+\cos(\pi-\phi_2(\gamma)+\theta_v))<T^2,\\
\label{11082007}& \frac{1-r_{v_m^{\square}}}{2r_{_{v_m^{\square}}}}(1+\cos(\pi-\phi_2(\gamma)+\theta_{v_m^{\square}})  \geq T^2,\\
\label{11082008}&  \frac{1-r_{v_l^{\square}}}{2r_{_{v_l^{\square}}}}(1+\cos(\pi-\phi_2(\gamma)+\theta_{v_l^{\square}}) < T^2.}
It is a direct computation that 
$$\gamma v_m^{\square}(1)=e^{\phi_1(\gamma)\bd{i}}\cdot\frac{\cosh\frac{t(\gamma)}{2}e^{\phi_2(\gamma)\bd{i}}v_m^{\square}(1)+\sinh \frac{t(\gamma)}{2}}{\sinh\frac{t(\gamma)}{2}e^{\phi_2(\gamma)\bd{i}}v_m^{\square}(1)+\cosh \frac{t(\gamma)}{2}}.$$
Therefore,  
\al{\label{11081825}d(\gamma v(1),\gamma v_m^{\square}(1))\sim 2e^{-t(\gamma)}\Big\vert\tan\left(\frac{\pi-\phi_2(\gamma)+\theta_v}{2}\right)- \tan\left(\frac{\pi-\phi_2(\gamma)+\theta_{v_m^{\square}}}{2}\right)\Big\vert,}
so the condition \eqref{110814474} is asymptotically equivalent to 
\al{ \label{11082010}
\frac{e^{t(\gamma)}}{2\Big\vert\tan\left(\frac{\pi-\phi_2(\gamma)+\theta_v}{2}\right)- \tan\left(\frac{\pi-\phi_2(\gamma)+\theta_{v_m^{\square}}}{2}\right)\Big\vert}\geq \frac{T^2}{s}.
}
We let ${\Omega}_{i,j,k}^{n}(w,\square,l,s; T)$ be the collection of $(t,\theta)$ which satisfies \eqref{02141724}, \eqref{11082009}, \eqref{11082006}, \eqref{11082007}, \eqref{11082008} and \eqref{11082010}.  
It's clear that ${\Omega}_{i,j,k}^{n}(w,\square,l,s; T)=2\log T+{\Omega}_{i,j,k}^{n}(w,\square,l,s; 0)$.\\

Now we can analyze the gap distribution function $F_{T,\mathcal{I}}(s)$ defined in \eqref{11082323}.
\begin{proof}[Proof of Theorem \ref{mainthm}]
First we drop out those sporadic terms, 
$$F_{T,\mathcal{I}}(s)=\frac{1}{2}\sum_{i,j,k}\sum_{w\in\mathcal{W}_{N_1}}^ {* }\sum_{\square=\{i,j\}}\sum_{l\leq N_{w}^{\square}}\sum_{n=0}^{\infty}\frac{\#\tilde{\Omega}_{i,j,k}^n(w,\square,l,s;T)}{c_oT^{2\delta}}+O\left(\frac{\log T}{T^{2\delta}}\right).$$
There's a factor $1/2$ because each gap is exactly counted twice.  The symbol $*$ means that the sum is restricted to those terms $w$ with $\text{Int}(w)\neq \rho_k$. We also drop all the terms with $n>M$ in the innermost sum of the above, for some big number $M$.  The contribution from these terms to $F_{T,\mathcal{I}}(s)$ is $O( {1}/{a^M})$ from Lemma \ref{11081149}.  We have
\al{
F_{T,\mathcal{I}}(s)=\frac{1}{2}\sum_{i,j,k}\sum_{w\in\mathcal{W}_{N_1}}^ {* }\sum_{\square=\{i,j\}}\sum_{l\leq N_{w}^{\square}}\sum_{n=0}^{M}\frac{\#\tilde{\Omega}_{i,j,k}^n(w,\square,l,s;T)}{c_oT^{2\delta}}+O\left(\frac{\log T}{T^{2\delta}}\right)+O\left(\frac{1}{a^{M}}\right).
}
Now we have a finite sum. Replacing $\tilde{\Omega}$ by $\Omega$ (since we are doing asymptotics) and applying Corollary \ref{11081155}, we obtain
\al{\nonumber\liminf_{T\rightarrow\infty} F_{T,\mathcal{I}}(s)\leq \frac{1}{2}\sum_{i,j,k}\sum_{w\in\mathcal{W}_{N_1}}^ {* }\sum_{\square=\{i,j\}}\sum_{l\leq N_{w}^{\square}}\sum_{n=0}^{M}\frac{S(\Omega_{i,j,k}^n(w,\square,l,s;1))}{c_0} -\frac{c_1}{a^{M}}\\
 \leq \frac{1}{2}\sum_{i,j,k}\sum_{w\in\mathcal{W}_{N_1}}^ {* }\sum_{\square=\{i,j\}}\sum_{l\leq N_{w}^{\square}}\sum_{n=0}^{M}\frac{S(\Omega_{i,j,k}^n(w,\square,l,s;1))}{c_0} +\frac{c_2}{a^{M}}\leq \limsup_{T\rightarrow\infty}F_{T,\mathcal{I}}(s),
}
for some positive absolute constant $c_1,c_2$.
Let $M$ go to $\infty$, we obtain 
\al{\lim_{T\rightarrow\infty}F_{T,\mathcal{I}}(s)=\frac{1}{2} \sum_{i,j,k}\sum_{w\in\mathcal{W}_{N_1}}^ {* }\sum_{\square=\{i,j\}}\sum_{l\leq N_{w}^{\square}}\sum_{n=0}^{\infty}\frac{S(\Omega_{i,j,k}^n(w,\square,l,s;1))}{c_0}.}

Thus the existence of the limiting distribution is proved.  \\

We can define $\tilde{\Omega}_{i,j,k}^n(w,\square,l,\infty;T)$, ${\Omega}_{i,j,k}^n(w,\square,l,\infty;T)$ in the same way as $\tilde{\Omega}_{i,j,k}^n(w,\square,l,s;T)$,
${\Omega}_{i,j,k}^n(w,\square,l,s;T)$ except removing conditions \eqref{110814474},\eqref{11082010}.  An identical argument as above shows that 
\al{\frac{1}{2}\sum_{i,j,k}\sum_{w\in\mathcal{W}_{N_1}}^ {* }\sum_{\square=\{i,j\}}\sum_{l\leq N_{w}^{\square}}\sum_{n=0}^{\infty}\frac{S(\Omega_{i,j,k}^n(w,\square,l,s;1))}{c_0}=1.}
Since each $S(\Omega_{i,j,k}^n(w,\square,l,s;1))$ is continuous, and 
$$\lim_{s\rightarrow\infty}S(\Omega_{i,j,k}^n(w,\square,l,s;1))=S(\Omega_{i,j,k}^n(w,\square,l,\infty;1)),$$
we obtain $F(s)$ is continuous and $$\lim_{s\rightarrow\infty}F(s)=1.$$

To show that $F$ is supported away from 0, we give a universal positive constant $s_0$ such that $S(\Omega_{i,j,k}^n(w,\square,l,s;1))=0$ when $s<s_0$, for all $i,j,k,n$.  From \ref{11082010} it's enough to give a universal upper bound for 
\al{\label{11091322}\frac{e^{t(\gamma)}}{2\Big\vert\tan\left(\frac{\pi-\phi_2(\gamma)+\theta_v}{2}\right)- \tan\left(\frac{\pi-\phi_2(\gamma)+\theta_{v_m^{\square}}}{2}\right)\Big\vert T^2}}
From \eqref{11082006}, 
\al{\nonumber\eqref{11091322}&\leq \frac{2r_{v}}{(1-r_{v})(1+\cos(\pi-\phi_2(\gamma)+\theta_v))\Big\vert\tan\left(\frac{\pi-\phi_2(\gamma)+\theta_v}{2}\right)- \tan\left(\frac{\pi-\phi_2(\gamma)+\theta_{v_m^{\square}}}{2}\right)\Big\vert }\\
&\label{11091346}\ll      \frac{r_{v}}{\Big\vert\tan\left(\frac{\pi-\phi_2(\gamma)+\theta_v}{2}\right)- \tan\left(\frac{\pi-\phi_2(\gamma)+\theta_{v_m^{\square}}}{2}\right)\Big\vert }
,}
since $\cos(\pi-\phi_2(\gamma)+\theta_v)$ stays uniformly away from $-1$ as $\text{End}(\gamma)\neq \text{Int}(v)$.   \\

We again only analyze two typical cases.  If $\{j,k\}=\{1,3\}$, then 
$$\Big\vert\tan\left(\frac{\pi-\phi_2(\gamma)+\theta_v}{2}\right)- \tan\left(\frac{\pi-\phi_2(\gamma)+\theta_{v_m^{\square}}}{2}\right)\Big\vert \gg d(\theta_v,\theta(v_m^{\square}))\gg ||(\rho_1\rho_2)^n||^{-1}d(\theta_w,\theta_{u_m^{\square}})$$
From Lemma \ref{11081149} $$r_{v}=||(\rho_1\rho_2)^nw||^{-1}\ll ||(\rho_1\rho_2)^n ||^{-1}a^{-n}.$$
Therefore, we have 
$$\eqref{11091346}\ll 1 ,$$
since we have only finitely many possible choices for $w$ and $u_m^{\square}$.\\
If  $\{j,k\}\neq \{1,2\}$, there's an extra situation when $\theta_v\in\mathfrak{m}_1$ and $\theta_{v_m^{\square}}\in \mathfrak{m}_2$. In this situation
$$\Big\vert\tan\left(\frac{\pi-\phi_2(\gamma)+\theta_v}{2}\right)- \tan\left(\frac{\pi-\phi_2(\gamma)+\theta_{v_m^{\square}}}{2}\right)\Big\vert \gg 1,$$
so we also have
$$\eqref{11091346}\ll 1 .$$
Therefore, such $s_0$ exists, so that $F$ is supported away from 0.  
\end{proof}

\appendix
\section{A random point process on the limit set} \label{AppendixA}
We prove the limiting behavior of a random point process on the limit set $\Lambda_\Gamma$.  The argument is an adaption from the appendix of \cite{BSR12} where they prove the limiting nearest spacing distribution of a random point process on $\mathbb{S}^n$ is Poisson.\\

Let $\tilde{d}(x,y)$ be the counterclockwise distance from $x$ to $y$ on $\partial \DD$ (so $\tilde{d}(\cdot,\cdot)$ is not symmetric, and $\tilde{d}(x,y)=2\pi-\tilde{d}(y,x)$ if $x\neq y$).  Let $I_r (\cdot, \cdot)$ be the indicate function such that $I_r (x,y)=1$ if $d(x,y)\leq r$ and $I_r (x,y)=0$ otherwise.  Let $P_1,P_2,..., P_N$ be $N$ independent random variables distributed according to the Patterson Sullivan measure $\mu$ on $\partial \DD$ and $d_1,d_2,\cdots,d_N$ be the length of the oriented gaps with starting starting points $P_1,P_2,\cdots, P_N$, respectively. Note that these $d_1,d_2,\cdots, d_N$ are not independent.  We expect the $d_1,d_2,\cdots, d_N$ are of the scale $N^{-\frac{1}{\delta}}$, so we define the distribution of scaled gap distribution in the following way: 
\al{\nu_N(s)=\nu(P_1,\cdots,P_N) (s):=\sum_{i=1}^{N}\frac{1}{N}\bd{1}_{s=d_iN^{\frac{1}{\delta}}}}
Let $s\geq 0$, we examine the expectation 
\al{\ee\left(\nu_N([0,s]) \right)=\ee\left( \frac{1}{N}\sum_{i=1}^N \bd{1}\{d_i N^{\frac{1}{\delta}}\leq s\} \right)\\
=\ee\left(\frac{1}{N}\sum_{i=1}^N\bd{1}\left\{\min_{j\neq i}\{\tilde{d}(P_i,P_j)\}\leq \frac{s}{N^{\frac{1}{\delta}}}\right\}\right)
}
  
We define 
\al{A_k (P_1,P_2,\cdots,P_N):= \frac{1}{N}\sum_{i=1}^N\sum_{\substack{j_i,\cdots,j_k\neq i\\j_1,\cdots,j_k\text{ distinct}}} \prod_{l=1}^k\bd{1}\{\tilde{d}(P_i,P_{j_l})\leq \frac{s}{N^{\frac{1}{\delta}}}\}}
and let $B_n=\sum_{i=1}^n(-1)^{i+1}A_i$.
Using inclusion-exclusion, we have 
\al{\ee\left(\nu_N([0,s]\right) \leq \ee(B_{2n+1})=\ee\left(A_1-A_2+\cdots+A_{2n+1}\right)\label{upperbound}}
and 
\al{\ee\left(\nu_N([0,s]\right) \geq \ee(B_{2n})=\ee\left(A_1-A_2+\cdots-A_{2n}\right)\label{lowerbound}}
For each $k$,
\al{\ee(A_k)=&\ee\left( \frac{1}{N}\sum_{i=1}^N\sum_{\substack{j_i,\cdots,j_k\neq i\\j_1,\cdots,j_k\text{ distinct}}} \prod_{l=1}^k\bd{1}\{d(P_i,P_{j_l})\leq \frac{s}{N^{\frac{1}{\delta}}}\}\right)\nonumber\\
&=\frac{(N-1)(N-2)\cdots(N-k)}{k!}\int_{\partial\DD}\mu\left(L\left(x,\frac{s}{N^{\frac{1}{\delta}}}\right)\right)^kd\mu(x),\label{02062238}
}
where $L(x,\eta)$ is the interval $\{y\in\partial \DD|\tilde{d}(x,y)\leq \eta\}$.

Now according to Proposition 3 and the comment in Section 3 from \cite{Su79}, there exists a positive constant $\eta_0$ and two positive constants $0<c<C$ such that if $x\in \Lambda_\Gamma$ and $\eta<\eta_0$, we have universal upper and lower bounds for the ratio $p(x,\eta)=\frac{\mu(L(x,\eta))}{\eta^{\delta}}$:
\al{c<p(x,\eta)<C\label{02070907}}
This can not be improved (i.e, the limit of $p(x,\eta) $ when $\eta\rightarrow 0$ does not exist for most points) according to Corollary 4.10 from \cite{Fa86}.  \\

We rewrite \eqref{02062238} as 
\al{\label{02070129}\ee(A_k)=   \frac{(N-1)(N-2)\cdots(N-k)s^{\delta k}}{N^kk!}\int_{\partial\DD}\left(p\left(x,sN^{-\frac{1}{\delta}}\right)\right)^kd\mu(x) }\\
As $N$ goes to infinity, 
\al{\label{02070135}\ee(A_k)\rightarrow \frac{s^{\delta k}}{k!}\int_{\partial\DD}\left(p\left(x,sN^{-\frac{1}{\delta}}\right)\right)^kd\mu(x)}

We can also give an upper bound for $\text{Var}(A_k)$, the variance of $A_k$.  Using Cauchy-H\"older inequality, 
\al{\nonumber\text{Var}(A_k)&=\ee\left(\left(A_k-\ee(A_k)\right)^2\right)\nonumber\\
&\leq  \frac{1}{N^2k!}\sum_{i=1}^N\sum_{\substack{j_i,\cdots,j_k\neq i\\j_1,\cdots,j_k\text{ distinct}}} \text{Var}\left( \prod_{l=1}^k\bd{1}\{d(P_i,P_{j_l})\leq \frac{s}{N^{\frac{1}{\delta}}}\}\right)\nonumber\\
&\leq \frac{s^{\delta k}}{Nk!} \int_{\partial\DD}\left(p\left(x,sN^{-\frac{1}{\delta}}\right)\right)^kd\mu(x) \left(1-\frac{s^{\delta k}}{N^k}\int_{\partial\DD}\left(p\left(x,sN^{-\frac{1}{\delta}}\right)\right)^kd\mu(x)\right)\nonumber\\
&\leq \frac{s^{\delta k}}{Nk!} \int_{\partial\DD}\left(p\left(x,sN^{-\frac{1}{\delta}}\right)\right)^kd\mu(x) \leq \frac{(s^{\delta}C)^k}{Nk!}
}

Therefore, for any $N$, we have
\al{\sum_{k=1}^\infty \text{Var}(A_k)\leq \frac{e^{s^{\delta}C}}{N}}
From \eqref{02070135}, \eqref{upperbound} and \eqref{lowerbound},  we know that 
\al{\ee(\nu_N([0,s]))\rightarrow 1-\int_{\partial\DD}e^{-s^{\delta} p\left(x,{s}{N^{-\frac{1}{\delta}}}\right)}d\mu(x):= 1-Z_N(s) }
We give an upper bound of $\text{Var}(\nu_N([0,s]))$ in the following way:
\al{
\text{Var}(\nu_N([0,s]))=&\ee\left((\nu_N([0,s])-Z_N(s))^2\right)\nonumber\\
\leq& \ee\left((B_{2n+1}-Z_N(s))^2\right)+\ee\left((B_{2n}-Z_N(s))^2\right)\nonumber\\
\ll &\ee\left((B_{2n+1}-\ee\left(B_{2n+1}\right))^2\right)+ \ee\left((B_{2n}-\ee\left(B_{2n}\right))^2\right)\nonumber\\&+\left(\ee\left(B_{2n+1}\right)-Z_N(s)\right)^2+\left(\ee\left(B_{2n}\right)-Z_N(s)\right)^2\nonumber\\
\ll&\frac{e^{Cs^{\delta}}}{N}+\frac{n^2}{N}+\frac{e^{Cs^{\delta }(2n+1)}}{2n!}\label{02070855}
}

Set $n=[\log N]$, the integral part of $\log N$, then from \eqref{02070855}, as $N$ goes to infinity, $\text{Var} (\nu_N([0,s]))$ goes to 0.  A direct application of Chebyshev's inequality leads to the following theorem:
\begin{thm}\label{limitingdistribution}
As $N$ goes to infinity, $ \nu_N([0,s])$ converges to $Z_N(s)$ almost surely.  
\end{thm}

Unlike the case of manifolds,  we do not know if the limiting distribution exists because of the local turbulence of the Patterson-Sullivan measure (see \eqref{02070907}).   But from Theorem \ref{limitingdistribution}, we know that if we scatter $N$ random points independently on $\Lambda_\Gamma$,  when $N$ large almost surely most gaps are of the order  $\frac{1}{N^{\frac{1}{\delta}}}$, and $\nu_N([s,\infty])\leq e^{-cs^{\delta}}$, so almost surely $\nu_N$ has a tail of exponential decay.

\bibliographystyle{plain}
\bibliography{thingap}

\end{document}